\documentclass{amsart}
\usepackage{amssymb,latexsym}
\usepackage{amsmath}
\usepackage{graphicx}
\usepackage{amscd}
\usepackage{color}
\usepackage{enumerate}
\newenvironment{enumeratei}{\begin{enumerate}[\upshape (i)]}{\end{enumerate}}
\numberwithin{equation}{section}
\theoremstyle{plain}
 \newtheorem{theorem}{Theorem}[section]
 \newtheorem{lemma}[theorem]{Lemma}

 \newtheorem{corollary}[theorem]{Corollary}
 \newtheorem{observation}[theorem]{Observation}
\theoremstyle{definition}
 \newtheorem{definition}[theorem]{Definition}
 \newtheorem{remark}[theorem]{Remark}

\newcommand\bomegakern{\kern 2pt}

\newcommand\whot [1] {#1^\ast}
\newcommand\url [1] {{\texttt{#1}}}

\newcommand \datum {\hfill 04:50,  November 17, 2017}

\renewcommand\rho{\varrho}

\newcommand \Con  {\textup{Con}}
\newcommand \Princ  {\textup{Princ}}
\newcommand \Aut  {\textup{Aut}}
\newcommand \Sub  {\textup{Sub}}
\newcommand \Id  {\textup{Id}}
\newcommand \Filt  {\textup{Filt}}
\newcommand \depth  {\textup{depth}}

\newcommand \fvar [1] {\mathcal V_{#1}}
\newcommand \modvar [1] {{\pfie #1\textup{-\tbf{Mod}}}}
\newcommand \pfie [1] {Q_{#1}}
\newcommand \vsp [2] {V_{#1}(#2)}
\newcommand \nnn {\mathbb N_0}
\newcommand \nnp {\mathbb N^*}
\newcommand \clat [2] {\textup{Sub}(\vsp{#1}{#2})}
\newcommand \fla[2] {L_{#1}(#2)} 
\newcommand \ula[2] {L^{\kern-1pt\sst{+}}_{#1}(#2)} 
\newcommand \sst [1] {{\scriptscriptstyle{#1}}}
\newcommand \nablaell[1] {{\nabla_{\kern-3pt #1}}}
\newcommand \deltaell[1] {{\Delta_{\kern-1pt #1}}}
\newcommand \Kfilter {\filter_{\kern-1pt K}}

\newcommand \restrict[2] {{#1\kern0.5pt\rceil_{\kern-1.5pt #2}}}
\newcommand\ideal[1]{\mathord\downarrow #1}
\newcommand\filter[1]{\mathord\uparrow #1}
\newcommand \tuple [1] {\langle #1\rangle}
\newcommand \pair [2] {\tuple{#1,#2}}

\newcommand \tbf [1] {\textbf{#1}} 
\newcommand \set[1] {\{#1\}}
\newcommand \red [1] {\color{red}#1\color{black}}

\renewcommand\phi{\varphi}
\renewcommand\epsilon{{\boldsymbol\varepsilon}}
\newcommand\semmi [1] {}

\begin{document}
\title
[Congruences of a lattice with more ideals than filters]
{On principal congruences and the number of  congruences of a lattice with more ideals than filters}

\author[G.\ Cz\'edli]{{G\'abor Cz\'edli}}
\email{czedli@math.u-szeged.hu}
\urladdr{http://www.math.u-szeged.hu/~czedli/}
\address{Bolyai Institute, University of Szeged, Hungary 6720}

\author[C.\ Mure\c san]{Claudia Mure\c san}
\email{c.muresan@yahoo.com, cmuresan@fmi.unibuc.ro}
\address{Faculty of Mathematics and Computer Science of the University of Bucharest,  RO 010014,  Romania;  and also:  Universit\`a degli Studi di Cagliari }

\thanks{This research was supported by
NFSR of Hungary (OTKA), grant number K 115518, and by  the research grant Propriet\`a d'Ordine Nella Semantica Algebrica delle Logiche Non–
classiche of Universit\`a degli Studi di Cagliari, Regione Autonoma della Sardegna, L. R. 7/2007, n. 7, 2015, CUP: F72F16002920002.}

\dedicatory{Dedicated to the memory of George Allen Hutchinson}

\begin{abstract} 
Let $\lambda$ and $\kappa$ be cardinal numbers such that $\kappa$ is infinite and either $2\leq \lambda\leq \kappa$, or $\lambda=2^\kappa$. 
We prove that there exists a lattice $L$ with exactly $\lambda$ many congruences, $2^\kappa$ many ideals, but only  $\kappa$ many filters. Furthermore, if $\lambda\geq 2$ is an integer of the form $2^m\cdot 3^n$, then 
we can choose $L$ to be a  modular 
lattice generating one of the  minimal modular nondistributive congruence varieties described by Ralph Freese in 1976, and this $L$ is even relatively complemented for $\lambda=2$. Related to some earlier results of 
George Gr\"atzer and the first author, we also prove that if $P$ is a bounded ordered set (in other words, a bounded poset) with at least two elements, $G$ is a group, and $\kappa$ is an infinite cardinal such that $\kappa\geq |P|$ and $\kappa\geq |G|$, then there exists a lattice $L$ of cardinality $\kappa$ such that (i) the principal congruences of $L$ form an ordered set isomorphic to $P$, (ii) the automorphism group of $L$ is isomorphic to $G$, (iii) $L$ has $2^\kappa$ many ideals, but (iv) $L$ has only $\kappa$ many filters.   
\end{abstract}

\subjclass {06B10 {{\color{red}
\datum{}\color{black}}}}
\keywords{Lattice ideal, lattice filter, simple lattice, more ideals than filters,  number of ideals,  cardinality, lattice congruence, principal congruence}

\maketitle
\section{Introduction and result}
For a lattice $L$, let $\Con(L)$, $\Filt(L)$, 
and $\Id(L)$ stand for the lattice of congruences, that of filters, and that of ideals of $L$, respectively. 
Motivated by Mure\c san~\cite{muresan2017arXiv}, we say that a triple $\tuple{\alpha,\beta,\gamma}$ of cardinal numbers is \emph{CFI-represented by a lattice} $L$ if $\tuple{\alpha,\beta,\gamma}= \tuple{\,|\Con(L)|,|\Filt(L)|,|\Id(L)|\,}$. Also, we say that  $\tuple{\alpha,\beta,\gamma}$ is an \emph{eligible triple} (of cardinal numbers) if 
there exists an infinite cardinal number $\delta$ such that $2\leq \alpha\leq 2^\delta$ and $\delta\leq \beta<\gamma\leq 2^\delta$. 
A lattice $L$ always has $\delta:=|L|$ many principal filters and principal ideals, and only those filters and ideals if $L$ is finite. Hence,  cardinal arithmetics trivially implies that 
\begin{equation}
\parbox{9cm}{if a triple  is CFI-represented by such a nonsingleton lattice that has \emph{more ideals than filters}, then this triple is eligible.}
\end{equation}
This raises the question whether every eligible triple is CFI-representable. Although we cannot answer this question in full generality, it follows trivially from our first theorem
to be stated soon that the answer is affirmative under the generalized continuum hypothesis. We are also interested in whether eligible triples can be represented by   ``nice'' lattices but we can show some  ``beauty'' of these lattices 
 only for certain eligible triples.
Postponing the  definition of $\fvar p$ until Definition~\ref{deffvar},  note that a lattice $L\in \fvar p$ \emph{generates} $\fvar p$ iff $L$ satisfies exactly the same lattice identities that are satisfied by all members of $\fvar p$.

\begin{theorem}\label{thmmain}
For every infinite cardinal number $\kappa$,
the following three statements hold.
\begin{enumeratei}
\item\label{thmmaina} If $2\leq \lambda\leq \kappa$, then the triple $\tuple{\lambda,\kappa, 2^\kappa}$ is CFI-representable.
\item\label{thmmainb} The triple $\tuple{2^\kappa,\kappa, 2^\kappa}$ can be CFI-represented by a \emph{distributive} lattice.
\item\label{thmmainc} Let $\fvar p$ be one of the minimal modular nondistributive congruence varieties described by Freese~\cite{freeseabstr}; see Definition~\ref{deffvar} later. If $m$ and $n$ are nonnegative integers with $m+n\geq 1$, then the triple $\tuple{2^m\cdot 3^n ,\kappa, 2^\kappa}$ can be CFI-represented by a \emph{modular} lattice $L=L(p,2^m\cdot 3^n ,\kappa)\in \fvar p$ such that $L$  generates $\fvar p$. For  $\pair m n=\pair 1 0$, this lattice  is relatively complemented.
\end{enumeratei}
\end{theorem}

\begin{remark}\label{remarkone} The case $\lambda=2$ belongs to the scope of \ref{thmmain}\eqref{thmmainc} and provides a \emph{simple}, relatively complemented, modular lattice with more ideals than filters. Since an infinite distributive lattice $L$ always has at least
$|L|$ many congruences by the prime ideal theorem and the lattice in \ref{thmmain}\eqref{thmmainc}  is necessarily infinite, $\fvar p$ in  \ref{thmmain}\eqref{thmmainc} cannot be replaced by the variety of distributive lattices. However, we do not know whether $\fvar p$ can be replaced by a smaller lattice variety. 
\end{remark}

A congruence is \emph{principal} if it is generated by a single pair of elements. For a lattice $L$, let $\Princ(L)=\tuple{\Princ(L);\subseteq}$ denote the ordered set of \emph{principal congruences} of $L$.
In his pioneering paper, Gr\"atzer~\cite{ggprincrep} proved that, up to isomorphism,  every bounded ordered set is of the form $\Princ(L)$ for a bounded lattice $L$. This result was soon followed by several related results proved in  Cz\'edli~\cite{czgsingleinjectiveprinc,czgprincc,czgautprinc,czginjlatcat,czgcometic} and Gr\"atzer~\cite{ggprincrep,gghomoprinc1,gghomoprinc2,gghomoprinc3}.  Here, we add another related result since, as a by-product of the proof of \ref{thmmain}\eqref{thmmaina}, we have found the following statement, which is stronger than part \eqref{thmmaina} of Theorem~\ref{thmmain}. Recall that two element intervals are called \emph{prime intervals} and ordered sets are also called \emph{posets}.

\begin{theorem}\label{thmscnd} 
Let $P$ be a bounded ordered set with at least two elements and let $G$ be an arbitrary group. Then, for every infinite cardinal $\kappa$ such that $\kappa\geq\max\set{|P|, |G|}$, there exists a lattice $L$ with the following four properties:
\begin{enumerate}[\upshape(a)]
\item\label{thmscnda} $|\Id(L)|=2^\kappa$ but $|\Filt(L)|=\kappa=|L|$, so $L$ has more ideals than filters;
\item\label{thmscndb} $\Princ(L)$ is isomorphic to $P$; 
\item\label{thmscndc} every principal congruence of $L$ is generated by a \emph{prime} interval; 
\item\label{thmscndd} the automorphism group  $\Aut(L)$  of $L$ is isomorphic to $G$. 
\end{enumerate}
\end{theorem}

Note that while Cz\'edli~\cite{czgautprinc} gives a \emph{selfdual} lattice $L$ to represent $P$ and $G$ simultaneously,  this theorem yields, in some sense, the ``most non-selfdual'' $L$ for the same purpose. 
We take the opportunity to include the following observation, which gives the selfdual variant of a result of Baranski\v\i~\cite{baranskij84} and Urquhart~\cite{urquhart}; this observation will be proved in few lines in Section~\ref{sectionproof}.

\begin{observation}\label{remsdLdaut}
For every \emph{finite} distributive lattice $D$ and every \emph{finite} group $G$, there is a finite \emph{selfdual} lattice $L$ such that $D\cong\Con(L)$ and $G\cong\Aut(L)$. 
\end{observation}

As a candidate for the ``most non-selfdual'' variant, we present the following by-product  of
the \emph{proof} of Theorem~\ref{thmscnd}.

\begin{corollary}\label{corolsdhhGrt}
For every \emph{finite} distributive lattice $D$ and every group $G$, if $\kappa$ is an infinite cardinal with $\kappa\geq |G|$, then there is a lattice $L$ with $2^\kappa$ many ideals but only $\kappa$ many filters such that $D\cong\Con(L)$ and $G\cong\Aut(L)$.
\end{corollary}

Note that Wehrung~\cite{wehrungbigresult} rules out the chance of dropping the adjective ``finite'' from this corollary since there are infinite algebraic distributive lattices that are not representable by congruence lattices of lattices. 
Now we give more details on $\fvar p$ occurring in Theorem~\ref{thmmain}.

\begin{definition}\label{deffvar}
A variety $\mathcal U$ of lattices is called a \emph{congruence variety} if there exists a variety $\mathcal W$ of general algebras such that $\mathcal U$ is generated by the congruence lattices of 
all members of $\mathcal W$; the reader may want but need not 
see J\'onsson~\cite{jonsson} for a survey. In the situation just described, $\mathcal U$ is the congruence variety \emph{determined by} $\mathcal W$. For a prime number $p$ or $p=0$, we denote by $\fvar p$ the congruence variety determined by the variety of all vector spaces over the prime field of characteristic $p$.
\end{definition}

According to a remarkable discovery of Nation~\cite{nation}, not every lattice variety is a congruence variety. It was observed by Freese~\cite{freeseabstr} and published with a proof in Freese, Herrmann, and Huhn~\cite[Corollary 14]{freesehh} that the congruence varieties $\fvar p$ for $p$ prime or zero are pairwise distinct \emph{minimal} nondistributive modular congruence varieties and there is no other such variety. Besides the modular law,
every $\fvar p$ and so all the $L(p,2^m\cdot 3^n,\kappa)$ in \ref{thmmain}\eqref{thmmainc} satisfy many lattice identities that are stronger than modularity. Since the lattice  $L(p,2^m\cdot 3^n,\kappa)$ generates $\fvar p$,
we conclude the following fact.

\begin{remark} For every prime $p$ or $p=0$ and for every lattice identity $\Gamma$,  
the algorithm given in 
Hutchinson and Cz\'edli~\cite{hutchczg}
is appropriate to decide whether $\Gamma$ holds in the lattice $L(p,2^m\cdot 3^n,\kappa)$ occurring in Theorem~\ref{thmmain}\eqref{thmmainc}. 
\end{remark}

\section{Proving the two theorems}\label{sectionproof}
The rest of the paper is devoted to the proofs
our statements formulated  in the previous section.
Our notation is standard in lattice theory. 
For concepts or notation that are neither defined, nor referenced here, see Gr\"atzer~\cite{ggCFL2};
see \texttt{tinyurl.com/lattices101} for its freely downloadable parts. Some familiarity with universal algebra is assumed; see, for example, Burris and Sankappanawar~\cite{bursan}; note that its Millennium Edition  is freely available from\\
\url{http://www.math.uwaterloo.ca/~snburris/htdocs/ualg.html}\ .

Recall that a lattice $M$ is said to satisfy the \emph{Descending Chain Condition} if whenever $x_0,x_1,x_2,\dots\in M$ and $x_0\geq x_1\geq x_2\geq\dots$, then $x_n=x_{n+1}=x_{n+2}=\dots$ for some $n\in\nnn$; see, for example, Gr\"atzer~\cite[page 105]{ggfoundbook}. 
Filters of the form $\filter a:=\set{x: x\geq a}$ are  \emph{principal filters}. 
The following easy statement belongs to the folklore; having no reference at hand, we are going to present a proof. 

\begin{lemma}\label{lemmadccfp}
If a lattice $K$ satisfies the Descending Chain Condition, then every filter of $K$ is a principal filter.
\end{lemma}

\begin{proof} Let $F$ be a filter of $K$.  Pick an element $f\in F$. 
Let us consider the set  $M:=\set{f\wedge g: g\in F}$. The assumption on $K$ implies that  $M$  has a minimal element $\whot f:=f\wedge \whot g$, where $\whot g\in F$. For later reference, note that 
\begin{equation}
\parbox{6cm}{the rest of the proof relies only on the fact that $M$ has a minimal element.}
\label{pbxMhsMlmNt}
\end{equation}
Since $\whot f\in M\subseteq F$, we have that $\filter \whot f\subseteq F$. Conversely, if  $
\whot f\nleq h$ held for some $h\in F$, then $\whot g\wedge h\in F$ would yield that 
 $\whot f> \whot f\wedge h=(f\wedge \whot g)\wedge h=f\wedge (\whot g\wedge h)\in M$, contradicting the minimality of $\whot f$ in $M$. This shows that $h\in F\setminus\filter \whot f$ is impossible, whence $F=\filter f$.
\end{proof}


Let $p$ be a prime or 0, and let $\pfie p$ denote the prime field of characteristic $p$. It is either $\mathbb Z_p$, the ring of modulo $p$ residue classes of integers, or $\mathbb Q$, the field of rational numbers. For a cardinal $\kappa$, finite or infinite, we often think of $\kappa$ as the set of all ordinal numbers $\iota$ with $|\iota|<\kappa$ without further warning.  For example, $4=\set{0,1,2,3}$ and $\aleph_0=\nnn$ is the set of nonnegative integers. In the present paper, $\iota$ and $\mu$ will \emph{always} denote ordinal numbers without explicitly saying so all the times. Similarly, $\kappa$ and $\lambda$ stand for cardinal numbers. 

Let $\vsp p \kappa$
denote the $\kappa$-dimensional vector space over $\pfie p$. It consists of all those (choice) functions $x\colon \kappa\to\pfie p$ for which 
$\{\iota: |\iota| <\kappa$ and $x(\iota)\neq 0\}$ is a finite set.
For  $|\mu|<\kappa$, the function $e_\mu\colon \kappa\to \pfie p$ is defined by $e_\mu(\mu)=1$ and, for $\iota\neq\mu$, $e_\mu(\iota)=0$. 
The set $\set{e_\iota: |\iota|<\kappa}$ is the \emph{natural basis} of $\vsp p\kappa$. The subspaces of $\vsp p\kappa$ form a complete lattice
$\clat p\kappa$. Using the well-known dimension equation $\dim(a\vee b)+\dim(a\wedge b)=\dim(a)+\dim(b)$ for $a,b\in\clat p\kappa$,  it follows that the finite dimensional subspaces of $\vsp p\kappa$ form a sublattice of $\clat p\kappa$; we denote this sublattice by $\fla p\kappa$.
Adding the top element $1_{\clat p\kappa} = \vsp p\kappa$ of $\clat p\kappa$ to $\fla p\kappa$, we obtain another lattice, $\ula p\kappa:=\fla p\kappa\cup\set{\vsp p\kappa}$.
If $\kappa$ is finite, then $\ula p\kappa = \fla p\kappa= \clat p\kappa$. If $\kappa$ is infinite, then $\ula p\kappa$ is a proper sublattice of $\clat p\kappa$,  $\ula p\kappa$ is a bounded lattice but $\fla p\kappa=\ula p\kappa\setminus\set{1_{\ula p\kappa}}$ has no largest element.

\begin{lemma}\label{lemmageneratVp}
Let $p$ be a prime number or $0$. Then for every infinite cardinal $\kappa$, each of $\fla p\kappa$ and $\ula p\kappa$ 
generates the variety $\fvar p$.
\end{lemma}

\begin{proof}
Since $\vsp p \kappa$ is a directed union of a system of copies of finite dimensional vector spaces $\vsp p n$, where $n\in\nnp:=\set{1,2,3,\dots}=\nnn\setminus\set 0$, 
\begin{equation}
\parbox{6cm}{
$\fla p\kappa$ is a directed union of sublattices isomorphic to $\fla p n$, where $n\in\nnp$.}
\label{pbxDirUn}
\end{equation} 
Hence, a lattice identity holds in $\fla p \kappa$ iff it holds in $\fla p n$ for all $n\in\nnp$. Let $\modvar p$ denote the variety of all vector spaces over $\pfie p$, and note that $\vsp p n$ is the free algebra on $n$ generators in this variety. 
Thus, using the canonical isomorphism between $\Con(\vsp p n)$ and $\fla p n=\Sub(\vsp p n)$, it follows that a lattice identity $\Gamma$ holds in $\fla p\kappa$ iff it holds in the congruence lattices of the free algebras of finite ranks in $\modvar p$. 
It is well known from the theory of congruence varieties and Mal'tsev conditions that this is equivalent to the satisfaction of $\Gamma$ in $\fvar p$; see, for example, Hutchinson and Cz\'edli~\cite{hutchczg}. Hence, $\fla p\kappa$ generates $\fvar p$. So does $\ula p\kappa$ since, as it is straightforward to see, 
$\Gamma$ holds in $\fla p\kappa$ iff it holds in $\ula p\kappa$.
\end{proof}

The \emph{glued sum}  $L_0\dot+ L_1$ 
of lattices $L_0$ with greatest element and 
$L_1$ with least element is their Hall--Dilworth gluing along $L_0\cap L_1=\set{1_{L_0}}=\set{0_{L_1}}$; see, for example, Gr\"atzer~\cite[Section IV.2]{ggfoundbook}.

\begin{lemma}\label{lemmaImtF}
If $p$ is a prime number or $0$,  $\kappa$ is an infinite cardinal, and $K$ is a lattice of finite length such that $|K|\leq \kappa$, then each of the glued sums $K\dot+\fla p\kappa$ and $K\dot+\ula p\kappa$ has $2^\kappa$ many ideals but only $\kappa$ many filters. So do $\fla p\kappa$ and $\ula p\kappa$. Furthermore, so does every lattice $L'$ that we obtain from $\fla p\kappa $ or $\ula p\kappa$ by replacing, for each atom $a$, the prime interval $[0,a]$ by a lattice $K(a)$ of finite length and size at most $\kappa$.
\end{lemma}

\begin{proof} 
Since $\vsp p\kappa$
is of cardinality $\kappa$ and it has $\kappa$ many finite subsets, it has $\kappa$ many finite dimensional subspaces. 
Hence, $|\fla p\kappa|=\kappa$ and we obtain that
\begin{equation}
|K\dot+\fla p\kappa| =  |K\dot+\ula p\kappa|
=|\ula p\kappa|= |\fla p\kappa|= \kappa.
\label{eqdmhmNVT}
\end{equation}
If we add a top element to a lattice that does not have a top, then the new lattice will have only one more filter and one more ideal. Thus, in the rest of the proof of the lemma, it suffices to deal only with $K\dot+\fla p\kappa$ and $L'$.
 
For a subset $X$ of the natural basis  
$\set{e_\iota: |\iota|<\kappa}$ of $\vsp p\kappa$, let $I(X)$ be the collection of all finite dimensional subspaces of the subspace $[X]$ spanned by $X$ in $\vsp p \kappa$. Clearly, $I(X)$ is an ideal of the lattice $\fla p\kappa$ and $I'(X):=I(X)\cup K$ is an ideal of $K\dot+\fla p\kappa$. If $X$ and $Y$ are distinct subsets of the natural basis, then there is an ordinal $\iota$ such that $|\iota|<\kappa$ and $e_\iota$ belongs to exactly one of $X$ and $Y$. Let, say, $e_\iota\in X\setminus Y$. Then the one-dimensional subspace $[e_\iota]$ belongs to $I'(X)\setminus I'(Y)$. Hence, the map from the power set of $\kappa$ to $\Id(K\dot+\fla p\kappa)$ defined by $X\mapsto I'(X)$ is injective. Therefore, $K\dot+\fla p\kappa$ has 
at least $2^\kappa$ many ideals, and we conclude by \eqref{eqdmhmNVT} that it has exactly $2^\kappa$ many ideals. Similarly, in case of $L'$, distinct ideals of $\fla p\kappa$ generate distinct ideals of $L'$, whereby $|\Id(L')|=2^\kappa$. 

If $a$ and $b$ are finite dimensional subspaces of $\vsp p\kappa$ and $a<b$, then $\dim(a)<\dim(b)$. This implies easily that 
\begin{equation}
\parbox{8.4cm}{each of the lattices $L'$, $\fla p\kappa$, $\ula p\kappa$, $K\dot+\fla p\kappa$, and $K\dot+\ula p\kappa$ satisfies the Descending Chain Condition.}
\label{pbxDChCnd}
\end{equation}
This fact, Lemma~\ref{lemmadccfp}, and \eqref{eqdmhmNVT} yield that  $K\dot+\fla p\kappa$ has exactly $\kappa$ many filters. The same holds for $L'$, because $|L'|=\kappa$ by elementary cardinal arithmetics. 
This proves the first and the third parts of Lemma~\ref{lemmaImtF}. The first part implies the second part trivially, since 
we can choose $K$ to be  the singleton lattice and then $\fla p\kappa = K\dot+\fla p\kappa$ and  
$\ula p\kappa = K\dot+\ula p\kappa$.
\end{proof}

\begin{lemma}\label{lemmaConLkpp}
If  $\kappa$ is an infinite cardinal and $p$ is a prime number or $0$, then $\fla p\kappa$ 
is a relatively complemented simple lattice while $\ula p\kappa$ has exactly three congruences.
\end{lemma}

Note that  $\ula p\kappa$ above  is neither complemented,
nor relatively complemented, because
\begin{equation}
\text{the top element of $\ula p\kappa$ for an infinite $\kappa$ is join irreducible.} 
\label{pbxToPirr}
\end{equation}
The smallest congruence and the largest congruence on a lattice $M$  will be denoted by $\Delta=\deltaell M$ and $\nabla=\nablaell M$, respectively.

\begin{proof}[Proof of Lemma~\ref{lemmaConLkpp}]
The subspace lattices of finite dimensional vector spaces are well known to be relatively complemented and simple; see, for example, Gr\"atzer~\cite[Theorem 392]{ggfoundbook} and
Wehrung~\cite[after Definition 7-5.1]{wehrungchptr}. Both properties are preserved by directed unions, whereby \eqref{pbxDirUn} yields the first part of the lemma. We conclude from \eqref{pbxToPirr} that
the equivalence $\Psi$ with blocks $\fla p\kappa$ and the singleton set $\set{1_{\ula p\kappa}}$ is a congruence on $\ula p\kappa$.  Let $\Theta\in \Con(\ula p\kappa)\setminus\set{\deltaell{\ula p\kappa},\Psi}$; we need to show that $\Theta=\nablaell{\ula p\kappa}$. 
First, assume that $\Theta$ collapses two distinct elements of $\fla p\kappa$. 
Then the restriction of $\Theta$ to the simple lattice  $\fla p\kappa$ is $\nablaell{\fla p\kappa}$, whereby $\Theta\geq \Psi$. Hence $\Theta =\nablaell{\ula p\kappa}$, because
$\nablaell{\ula p\kappa}$ is the only equivalence on $\ula p\kappa$ strictly above $\Psi$. Second, assume that $\tuple{x,1_{\ula p\kappa}}\in \Theta$ for some $x\in \fla p\kappa$. Using that $\fla p\kappa$ has no largest element and that the $\Theta$-block of $1_{\ula p\kappa}$ is a convex sublattice, 
we obtain an element $y\in \fla p\kappa$ such that $x<y$ and $\pair x y\in \Theta$. Thus, the first case applies, completing the proof of Lemma~\ref{lemmaConLkpp}.
\end{proof}

\begin{proof}[Proof of Theorem~\ref{thmscnd}]
Recall that 
a lattice is \emph{automorphism-rigid} if it has only one automorphism, the identity map. Remember that $\kappa=\set{\iota: |\iota|<\kappa}$. It is clear from Lemma 2.8 of
Cz\'edli~\cite{czgautprinc} and the construction 
used in its proof that  $|S(\iota)|\leq \max(\aleph_0,|G|)\leq \kappa$ for every  $\iota\in \kappa$. This fact and what we need from Lemma 2.8 of \cite{czgautprinc} can be summarized as follows: for every $\iota, \iota_1,\iota_2\in \kappa$,
\begin{equation}
\parbox{9.1cm}{
$S(\iota)$ is an automorphism-rigid and simple lattice of length 12, $|S(\iota)|\leq \kappa$, and $S(\iota_1)\cong S(\iota_2)$ implies that $\iota_1=\iota_2$.}
\label{pbxLmVnknN}
\end{equation}
Since every finite-dimensional subspace is the join of finitely many one-dimensional subspaces, each element of $\fla p\kappa$ is the join of finitely many atoms. This implies that every automorphism $\phi$ of $\fla p\kappa$ is determined by the restriction of $\phi$ to the set of atoms. For each atom $a$ of $\fla p\kappa$, choose an $\iota_a\in\kappa$ so that for distinct atoms $a$ and $b$ the  ordinals $\iota_a$ and $\iota_b$ should be distinct. 
Let $L'$ be the lattice we obtain from $\fla p\kappa$ by replacing the interval $[0,a]$ by $S(\iota_a)$ for each atom $a$ of $\fla p\kappa$. We claim that 
\begin{equation}
\text{$L'$ is automorphism-rigid.}
\label{txtLprmrigid}
\end{equation}
Recall that the \emph{height} and the \emph{depth} of an element $x$ in a bounded lattice is the length of the ideal $\ideal x=[0,x]$ and that of the filter $\filter x=[x,1]$, respectively. Let $\phi\colon L'\to L'$ be an automorphism. It preserves the heights of elements and the atoms of $\fla p\kappa$ are exactly the elements of height 12 in $L'$. 
Hence, $\phi$ maps the atoms of $\fla p\kappa$ to atoms of $\fla p\kappa$. 
Thus, if $a$ is an arbitrary atom of $\fla p\kappa$, then so is $\phi(a)$ and   
\begin{equation}
\phi(S(\iota_a))=\phi(\ideal_{L'}\, a) =   \ideal_{L'}\, \phi(a)
= S(\iota_{\phi(a)}).
\label{eqcTzGfmBqCPgZ}
\end{equation}
Hence, $S(\iota_a)\cong S(\iota_{\phi(a)})$, whereby \eqref{pbxLmVnknN} implies that $\phi(a)=a$ for all atoms of $\fla p\kappa$. That is, the restriction of $\phi$ to the set of atoms of $\fla p\kappa$ is the identity map. Since this  restriction determines the action of the lattice isomorphism  $\phi$ on $\fla p\kappa$, $\phi$ acts identically on $\fla p\kappa$.  
Furthermore, for all atoms $a$ of $\fla p\kappa$, \eqref{eqcTzGfmBqCPgZ} turns into $\phi(S(\iota_a))= S(\iota_a)$,  and so  the automorphism-rigidity part of \eqref{pbxLmVnknN} yields that $\phi$ acts identically on $S(\iota_a)$. 
Therefore, \eqref{txtLprmrigid} holds.

By  Cz\'edli~\cite[Theorem 1.1]{czgautprinc}, there exists a selfdual lattice $H$ such that
\begin{equation}
\parbox{7.5cm}{$\Aut(H)\cong G$, $H$ is of length 16, $\Princ(H)\cong P$, and if both $P$ and $G$ are finite, then so is $H$.}
\label{pbxdPdzrGjsBmQ}
\end{equation}
Moreover, the construction given in \cite{czgautprinc} makes it clear that
\begin{equation}
\parbox{9.7cm}{If $\Theta\in \Con(H)$ such that 
the $\Theta$-block of $0_H$ is distinct from the singleton $\set{0_H}$, then $\Theta=\nablaell H$.
Furthermore, $|H|\leq \kappa$ and  each principal congruence on $H$ is generated by a prime interval.}
\label{pbxlnLpBTp}
\end{equation}

\begin{figure}[ht] 
\centerline
{\includegraphics[scale=1.0]{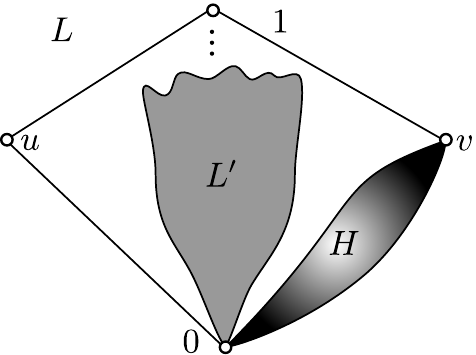}}
\caption{The lattice $L$ for Theorem~\ref{thmscnd}}
\label{figone}
\end{figure}

Now, we   define a lattice $L$ as follows;  see Figure~\ref{figone}. First, we add a top element and two additional elements, $u$ and $v$, to $L'$  such that $u$ and $v$ are complements of each other and of any nonzero element of $L'$. In the figure, $L'$ is indicated by uniform grey color. In the next step, replace the prime interval $[0,v]$ by $H$, identifying $0_H$ and $1_H$ with $0$ and $v$, respectively. In the figure, $H$ is radially colored. The lattice we have just described is $L$.
Note that $L'$ and $H$ are ideals of $L$ and every element of $H\setminus\set 0$ is a complement of every nonzero element of $L'\cup\set u$.

We will implicitly use the well-known fact that every congruence is determined by the pairs of \emph{comparable} distinct elements it collapses.
Let $\phi\colon\Con(L)\to \Con(H)$ be the restriction map defined by $\phi(\Theta):=\Theta\cap(H\times H)$; we claim that this map is bijective. 
First, we show that, for every $\Theta\in\Con(L)$,  
\begin{equation}
\parbox{6.4cm}{if $\Theta$ collapses a pair
$\pair x y\in L^2$ such that $x < y$ and
$\set{x,y}\nsubseteq H$, then $\Theta=\nablaell L$.}
\label{pbxcWpmlPwRbkVh}
\end{equation}
As a tool that we need for this, note that 
\begin{equation}
\parbox{6.1cm}{for every $z\in L'\setminus\set{0}$, the subset $\set{0,1,u,v,z}$
is a simple sublattice of $L$,
}
\label{pbxcSmlPsTMhr}
\end{equation}
since this sublattice is isomorphic to $M_3$.
We know from \eqref{pbxLmVnknN} that the lattices $S(\iota)$ used in the definition of $L'$ are simple. Hence, it follows from Lemma~\ref{lemmaConLkpp} that $L'$ is a simple lattice. So if $x,y\in L'\cup\set{u}$, then the simplicity of $L'$ gives that $\pair 0 z\in\Theta$, whence \eqref{pbxcSmlPsTMhr} yields that $\pair 0 1\in\Theta$, implying that $\Theta=\nablaell L$,
as required. The only remaining case is $x\leq v<y=1$; then $\pair v 1\in \Theta$ since the $\Theta$-blocks are convex. Thus, we obtain from
 \eqref{pbxcSmlPsTMhr} that  $\pair 0 1\in\Theta$, whereby $\Theta=\nablaell L$, proving \eqref{pbxcWpmlPwRbkVh}. Next, we claim that, for every $\Theta\in\Con(L)$, 
\begin{equation}
\phi(\Theta)=\deltaell H\quad\text{if and only if}\quad \Theta=\deltaell L.
\label{eqdhtZrsmP}
\end{equation}
The ``if'' part is obvious. Conversely, assume that $\phi(\Theta)=\deltaell H$ and $x<y$ in $L$. If $\set{x,y}\subseteq H$, then $\pair x y\notin \Theta$ since $\phi(\Theta)=\deltaell H$. If $\set{x,y}\nsubseteq H$, then  \eqref{pbxcWpmlPwRbkVh} gives that $\pair x y\notin\Theta$. Hence, $\Theta=\deltaell L$, proving \eqref{eqdhtZrsmP}.
Now, for the sake of contradiction, suppose that 
$\Theta_1, \Theta_2\in \Con(L)$ are distinct but $\phi(\Theta_1)=\phi(\Theta_2)$. We know from \eqref{eqdhtZrsmP} that none of $\Theta_1$ and $\Theta_2$ is $\deltaell L$. 
If one of them, say $\Theta_1$, collapses a pair described in  \eqref{pbxcWpmlPwRbkVh}, then $\Theta_1=\nablaell L$, so $\Theta_1$ and also $\phi(\Theta_2)=\phi(\Theta_1)$ collapses 
$\pair 0 v$, whence $\pair 0 v\in \Theta_2$ and \eqref{pbxcSmlPsTMhr} lead to $\Theta_2=\nablaell L =\Theta_1$, which is a contradiction. So none of $\Theta_1$ and $\Theta_2$ 
collapses a pair described in \eqref{pbxcWpmlPwRbkVh}. Hence, for every pair $\pair x y$ with $x<y$, if $\Theta_1$ or $\Theta_2$ collapses $\pair x y$, then $\pair x y\in H^2$. 
Using that the restrictions $\phi(\Theta_1)$ and $\phi(\Theta_2)$ coincide, it follows that $\Theta_1$ and $\Theta_2$ collapse the same pairs. Hence $\Theta_1=\Theta_2$,  contradicting our assumption and proving that $\phi$ is an injective~map.

Clearly, $\phi$ is order-preserving. If $\Psi\in\Con(H)\setminus \set{\nablaell H}$, then it follows easily from the first part of \eqref{pbxlnLpBTp} that $\Theta:=\deltaell L \cup \Psi$ is a congruence on $L$. Since its restriction to $H$ is $\Psi$, we have that $\Theta=\phi^{-1}(\Psi)$, while $\phi^{-1}(\nablaell H)=\nablaell L$ is clear. Thus, $\phi$ is a surjective map. Since $\phi^{-1}$ described above is order-preserving, we conclude that $\phi$ is a lattice isomorphism.
It is clear from \eqref{pbxcWpmlPwRbkVh} that if a comparable pair of elements generates a congruence distinct from the trivial congruences, then this pair is from $H^2$. Therefore,  the restriction of $\phi$ to $\Princ(L)$ is an order isomorphism from $\Princ(L)$ to $\Princ(H)$. Thus, \eqref{pbxdPdzrGjsBmQ} gives that $\Princ(L)\cong\Princ (H)\cong P$, proving part \eqref{thmscndb} of Theorem~\ref{thmscnd}.

Since $\phi$ is an isomorphism, the description of its inverse above, 
\eqref{pbxlnLpBTp}, and \eqref{pbxcWpmlPwRbkVh} imply that every principal congruence of $L$ is generated by a prime interval, in fact, by a prime interval of $H$. This proves  part \eqref{thmscndc} of Theorem~\ref{thmscnd}.

Next, denote the sublattice $L\setminus \set u$ by $L^{-u}$. Since $u$ is the only atom of $L$ that is also a coatom, $\psi(u)=u$ for every $\psi\in\Aut(L)$. Hence, the restriction from $L$ to $L^{-u}$ gives a group isomorphism from $\Aut(L)$ onto $\Aut(L^{-u})$.
So it suffices to focus only on $L^{-u}$.  By \eqref{pbxdPdzrGjsBmQ} and the sentence right above Lemma~\ref{lemmageneratVp} , $H=\{x\in L^{-u}: \depth(x)\leq 16$ or $x=0\}$. 
Similarly, $L'=\{x\in L^{-u}: \depth(x)\leq 16$ fails$\}$.
These first-order characterizations show that both $H$ and $L'$ are closed with respect to every $\psi\in\Aut(L)$. Thus, for $\psi\in\Aut(L)$, $\psi$ is determined by its restriction $\psi'$ to $L'$ and its restriction $\psi_H$ to $H$. But $\psi'$ is the identity map by \eqref{txtLprmrigid}, and it follows that the map $\Aut(L^{-u})\to \Aut(H)$, defined by $\psi\mapsto \psi_H$, is a group isomorphism. Hence,  \eqref{pbxdPdzrGjsBmQ} yields that $\Aut(L)\cong\Aut(L^{-u})\cong\Aut(H)\cong G$, proving part \eqref{thmscndd} of Theorem~\ref{thmscnd}.

Since $H$ is of length 16,  \eqref{pbxDChCnd} remains valid for $L$ instead of $\fla p\kappa$, that is, $L$ satisfies the Descending Chain Condition.  Elementary cardinal arithmetics based on \eqref{eqdmhmNVT} and $|H|\leq\kappa$ from \eqref{pbxlnLpBTp} shows that $|L|=\kappa$. Hence, it follows from Lemma~\ref{lemmadccfp} that $L$ has exactly $\kappa$ many filters. Every  ideal of $\fla p \kappa$ is an ideal of $L'$ and thus  of $L$, whereby Lemma~\ref{lemmaImtF} yields that $L$ has $2^\kappa$ many ideals. This yields  \eqref{thmscnda} and completes the proof of Theorem~\ref{thmscnd}.
\end{proof}

For a lattice $M$, $J(M)$ denotes the ordered set of nonzero join-irreducible elements of $M$.

\begin{proof}[Proof of Observation~\ref{remsdLdaut}]
Let $P=J(D)$. 
Take the \emph{finite selfdual} lattice $H$ from \eqref{pbxdPdzrGjsBmQ} and \eqref{pbxlnLpBTp}. 
It is well known from, say, Gr\"atzer~\cite[Page 39]{ggCFL2} that $\Theta\in J(\Con(H))$ iff $\Theta$ is generated by a prime interval. Hence, the second half of \eqref{pbxlnLpBTp} yields that  $J(\Con(H))=\Princ(H)\cong P=J(D)$. Thus, the structure theorem of finite distributive lattices, see Gr\"atzer~\cite[Corollary 2.15]{ggCFL2}, gives that $\Con(H)=D$. Letting $L:=H$, the rest follows from the choice of $H$.
\end{proof}

\begin{proof}[Proof of Corollary~\ref{corolsdhhGrt}]
Let $P:=J(D)$ and take $L$ from Theorem~\ref{thmscnd}. We have that $G\cong\Aut(L)$, $|\Id(L)|=2^\kappa$, $|\Filt(L)|= \kappa$, and $\Princ(L)\cong P$. Every congruence is the join of principal congruences, whereby $\Con(L)$ is finite since there are $|P|$, that is finitely many,  principal congruences. The same fact implies $J(\Con(L))\subseteq\Princ(L)$, since the above-mentioned join for a join-irreducible congruence contains only a single joinand. 
Assume that $\Theta\in\Princ(L)$ and $\Theta=\Psi_0\vee \Psi_1$ for $\Psi_1,\Psi_2\in\Con(L)$. By Theorem~\ref{thmscnd}\eqref{thmscndc}, $\Theta$ is generated by a pair $\pair a b\in L^2$ such that $a\prec b$. By the well-known description of joins, see Gr\"atzer~\cite[Theorem 1.2]{ggCFL2}, there is an $n\in\nnn$ and there are $x_0,x_1,\dots, x_n\in L$ such that $a=x_0\prec x_1\prec\dots\prec x_n=b$ and 
$\pair{x_{i-1}}{x_i}\in \Psi_0\cup\Psi_1$. 
Since $a\prec b$, we have that $n=1$ and $\pair a b=\pair {x_0}{x_1}\in \Psi_0\cup\Psi_1$, whereby $\Theta\leq \Psi_0$ or $\Theta\leq \Psi_1$. This implies that $\Theta\in J(\Con(L))$, so
$J(\Con(L))=\Princ(L)\cong P$, and we obtain 
$\Con(L)\cong D$ from the structure theorem of finite distributive lattices as in the previous proof. 
\end{proof}

\begin{lemma}\label{lemmagldsM}
Let $L_0$ and $L_1$ be as in the paragraph right before Lemma~\ref{lemmaImtF}.
Then $|\Con(L_0\dot+L_1)| = |\Con(L_0)|\cdot |\Con(L_1)|$. Also, a lattice identity holds in 
$L_0\dot+L_1$ iff it holds both in $L_0$ and in $L_1$.
\end{lemma}

\begin{proof}
The required equality follows from the straightforward fact that every $\Theta\in\Con(L_0\dot+ L_1)$ is determined by its restriction to $L_0$ and $L_1$, and these restrictions can be chosen arbitrarily. For $i\in\set{0,1}$, let $\Psi_i$ be the congruence of $L_0\dot+ L_1$ whose restriction to $L_i$ and $L_{1-i}$ are 
$\deltaell{L_i}$ and $\nablaell{L_{1-i}}$, respectively. Then $(L_0\dot+L_1)/\Psi_i\cong L_{i}$, for $i\in\set{0,1}$, and $\Psi_0\cap\Psi_1=\deltaell{L_0\dot+L_1}$. Hence, $L_0\dot+ L_1$ is (isomorphic to) a subdirect product of $L_0$ and $L_1$, and the second part of the lemma follows easily.
\end{proof}

Next, we prove our first theorem.

\begin{figure}[ht] 
\centerline
{\includegraphics[scale=1.0]{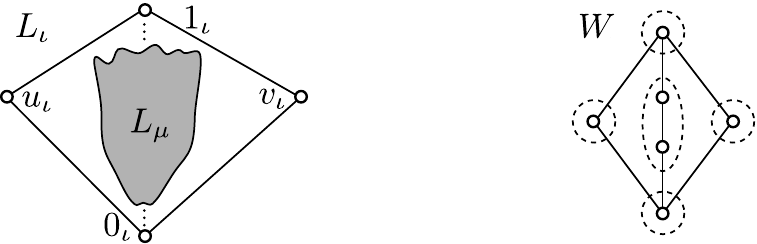}}
\caption{$L_\iota=L_{\mu+1}$ and the only nontrivial congruence on $W$}
\label{figtwo}
\end{figure}

\begin{proof}[Proof of Theorem~\ref{thmmain}]
We are going to give two different proofs for part \eqref{thmmaina}. The first one uses  Theorem~\ref{thmscnd}
and so relies on outer references. The second 
 one is self-contained modulo the paper.

For the first proof of part \eqref{thmmaina}, take a dually well-ordered bounded chain $P$ such that $|P|=\lambda$. Let $G$ be the one-element group, and take the lattice $L$ provided by Theorem~\ref{thmscnd}. It suffices to show that $|\Con(L)|=\lambda$.  Every $\Psi\in\Con(L)$ is a join of principal congruences, whence $\Psi$ is the join of the set $\set{\Theta: \Theta \in\Princ(L)\text{ and }\Theta\leq \Psi}$. Containing $\deltaell L$, this set is nonempty. Thus, $\Princ(L)\cong P$ implies that this set has a largest element $\Theta_0$. The join above is clearly $\Theta_0$, whence $\Psi=\Theta_0\in\Princ(L)$. Hence, $\Con(L)\subseteq \Princ(L)$, which gives that $\Princ(L)=\Con(L)$. Thus, $|\Con(L)|=|\Princ(L)|=|P|=\lambda$, completing the first proof of part \eqref{thmmaina}.

\begin{figure}[ht] 
\centerline
{\includegraphics[scale=1.0]{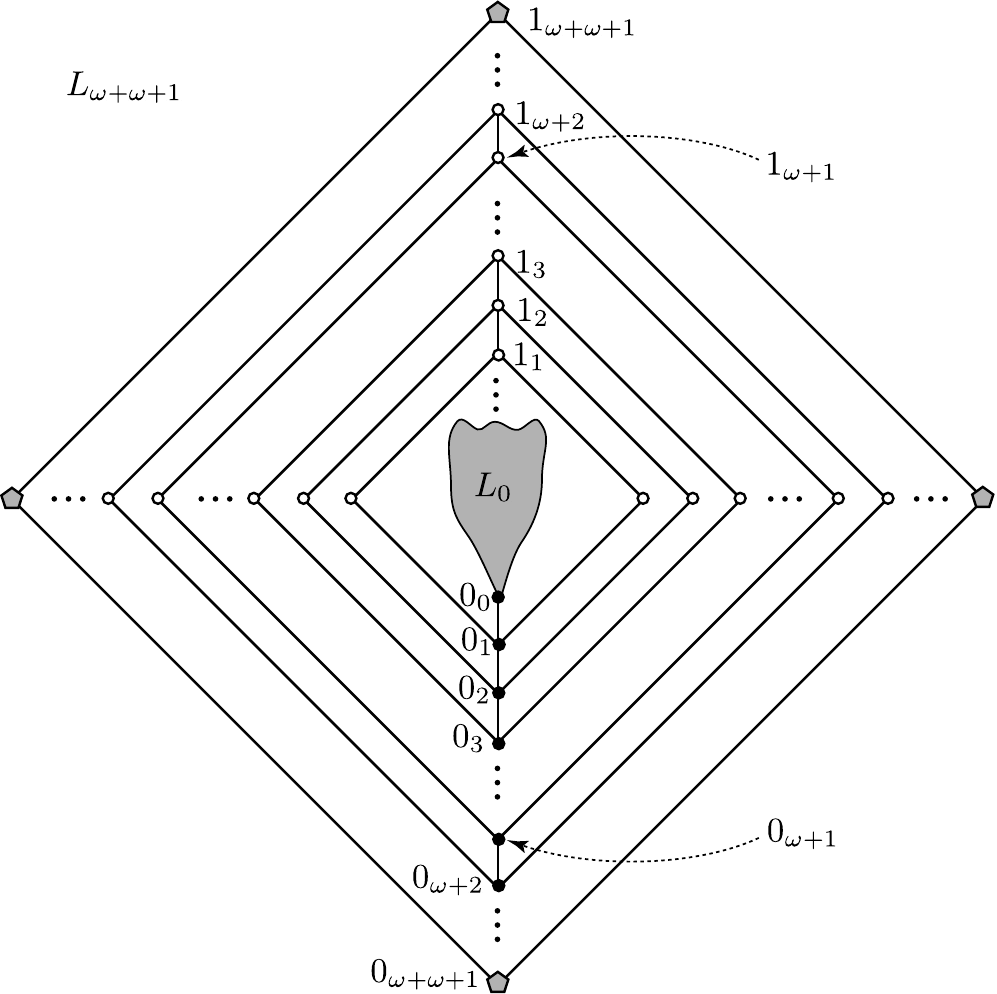}}
\caption{$L_{\omega+\omega+1}$, $C=\{$black-filled elements$\}$,  and $L_{\omega+\omega}=L_{\omega+\omega+1}\setminus\{$grey-filled pentagon-shaped elements$\}$}
\label{figthree}
\end{figure}

For the second proof of part \eqref{thmmaina}, 
let $\tau$ be the smallest ordinal with cardinality $\lambda$. We define a sequence  $\tuple{L_\iota: \iota\leq\tau}$ of lattices by induction as follows. We let $L_0=\fla p\kappa$. 
If $\iota$ is a successive ordinal of the (unique) form $\iota=\mu+1$,  then we obtain $L_\iota$ from $L_\mu$ by adding a new bottom and a new top first, and adding two incomparable new elements, $u_\iota$ and $v_\iota$, that are atoms and also coatoms; see Figure~\ref{figtwo} on the left and, for $\iota=\omega+\omega+1$ (and even for $\iota=\omega+\omega$), see Figure~\ref{figthree}. Note that a triple of three dots in Figure~\ref{figtwo} stands for an edge if $L_\mu$ is bounded but this is not always so; for example, neither $L_\omega$,  nor $L_{\omega+\omega}$ is bounded. In Figure~\ref{figthree}, the triples of three dots always stand for repetition.    If $\iota$ is a limit ordinal, then let $L_\iota$ be the union of the lattices $L_\mu$ for $\mu<\iota$; it is a lattice since we have formed a directed union.
We claim that $L:=L_\tau\,$ CFI-represents $\tuple{\lambda,\kappa,2^\kappa}$. It suffices to show by induction on $\iota$ that for all $\iota\leq \tau$, the condition 
\[
H(\iota):\qquad 
\parbox{6.8cm}{$L_\iota$  CFI-represents $\tuple{2+|\iota|,\kappa,2^\kappa}$, and every $\Theta\in\Con(L_\iota)\setminus\set{ \deltaell{L_\iota} }$
has exactly one non-singleton block, which is $L_\mu$ for some $\mu\leq\iota$}
\]
holds.
The validity of $H(0)$ follows from Lemmas~\ref{lemmaImtF} and \ref{lemmaConLkpp}. For the induction step, assume that $0<\iota$ and $H(\mu)$ holds for all $\mu$ that are less than $\iota$; we need to show the validity of $H(\iota)$. Depending on $\iota$, there are two cases. First, let 
 $\iota$ be a successor ordinal with $\iota=\mu+1$. Then, 
based on the congruence structure  of the lattice $W$ in Figure~\ref{figtwo}, the description of the congruences of $L_\iota$ follows easily, and we obtain that 
$|\Con(L_\iota)|= 1+|\Con(L_\mu)|=1+2+|\mu|=2+|\iota|$. 
Second,  let $\iota$ be a limit ordinal. 
For a nontrivial congruence $\Theta$ of $L_\iota$, that is, for    $\Theta\in\Con(L_\iota)\setminus\set{\deltaell{L_\iota}, \nablaell{L_\iota}}$, there exists a $\mu_0<\iota$ such that the restriction of $\Theta$ to $\mu_0$ is nontrivial. The induction hypothesis $H(\mu_0)$ yields a $\mu$, and it is straightforward to show that this $L_\mu$ is the only non-singleton block of $\Theta$. After having described the nontrivial congruences of $L_\iota$ by means of these ordinals $\mu$, we obtain $|\Con(L_\iota)|=|\iota|=2+|\iota|$ again.

It is straightforward to see that $|L_\iota|=\kappa$ and distinct ideals of $\fla p\kappa$ generate distinct ideals of $L_\iota$. These facts and Lemma~\ref{lemmaImtF} yield that $L_\iota$ has exactly $2^\kappa$ many ideals. Next, we generalize Lemma~\ref{lemmadccfp} as follows. 
\begin{equation}
\parbox{9.4cm}{Let $I$ be an ideal of a lattice $K$ and assume that whenever $x_0>x_1>\dots>x_n>x_{n+1}>\dots $ with $x_i\in K$ for all $n\in\nnn$, then  $x_n\in I$ for all but finitely many $n\in \nnn$. Then every non-principal filter of $K$ is generated by a filter of $I$.}
\label{pbxGnwgrlMmkg}
\end{equation}
Note that, by letting $I=\set{0}$, Lemma~\ref{lemmadccfp}  follows easily from \eqref{pbxGnwgrlMmkg}. In order to show \eqref{pbxGnwgrlMmkg}, let $F$ be a non-principal filter of $K$. For the sake of contradiction, suppose that $F\cap I=\emptyset$. Since $M$ in the proof of Lemma~\ref{lemmadccfp} is clearly a subset of $F$, $M\cap I=\emptyset$. Hence, there is no infinite strictly descending sequence in $M$, whereby $M$ has a minimal element and we conclude from \eqref{pbxMhsMlmNt} that $F$ is a principal filter of $K$. This is a contradiction and we obtain that $T:=F\cap I$ is nonempty. Clearly, $T$ is a filter of $I$, so we need to show only that $\Kfilter  T=F$. Fix a $t_0\in T$ and let $f$ be an arbitrary element of $F$. Since $t_0\wedge f\in F\cap I=T$ and 
$f\geq t_0\wedge f$, we have that $f\in \Kfilter T$. Hence, $F\subseteq \Kfilter T$ while the converse inclusion is trivial by $T\subseteq F$. Consequently, \eqref{pbxGnwgrlMmkg} holds.

Armed with \eqref{pbxGnwgrlMmkg}, let $C:=\set{0_\mu: \mu<\iota}$, where  $0_\mu$ denotes the bottom of $L_\mu$; see the set of the black-filled elements in Figure~\ref{figthree} for $\iota=\omega+\omega$. Note that $C$ is a chain and it witnesses that the Descending Chain Condition fails in $L_\iota$ in general. Observe that the chain $C$ is an ideal of $L_\iota$ and the assumptions of \eqref{pbxGnwgrlMmkg} hold for $\tuple{L_\iota,C}$ in place of $\tuple{K,I}$. Hence, $|\Filt(L_\iota)|  \leq |L_\iota|+|\Filt(C)|$. 
We have that $0_\mu < 0_{\mu'}$ iff $\mu>\mu'$. Hence, the chain $C$ is dually isomorphic to the chain $C':=\set{\mu: \mu<\iota}$ of ordinals, and so $|\Filt(C)|=|\Id(C')|$. Every proper ideal $J$ of $C'$ is uniquely determined by the smallest ordinal in $C'\setminus J$, whence it follows that $C'$ has at most $|C'|=|\iota|\leq \kappa$ many ideals. This fact, $|L_\iota|=\kappa$, and $|\Filt(L_\iota)|  \leq |L_\iota|+|\Id(C')|$ imply that $|\Filt(L_\iota)| =\kappa$. 
This completes the induction step and the second proof of 
part~\eqref{thmmaina}.

In order to prove \eqref{thmmainb}, let $L$ be the sublattice $\{x\in \nnn^\kappa: x(\iota)=0$ for all but finitely many $\iota$ with $|\iota|<\kappa\}$ of the $\kappa$-th direct power of  $\nnn=\tuple{\nnn;\leq}$. As a sublattice of a distributive lattice, $L$ is distributive.  For $X\subseteq \kappa$, we
let $\Theta(X)=\set{\pair x y\in L^2: x(\iota)= y(\iota) \text{ for all }\iota\in X}$ and $I(X)=\set{x(\iota)= 0\text{ for all }\iota\in X}$.
Using that $\Theta(X)\in\Con(L)$, $I(X)\in \Id(L)$, and $X_1\neq X_2$ implies that $\Theta(X_1)\neq\Theta(X_2)$ and $I(X_1)\neq I(X_2)$, we obtain
that $|\Con(L)|=|\Id(L)|=2^\kappa$. Since $L$ satisfies the Descending Chain Condition, Lemma~\ref{lemmadccfp} gives that $|\Filt(L)|=\kappa$. This proves part~\eqref{thmmainb}.

Finally, we turn our attention to part  \eqref{thmmainc}. If $\pair  m n=\pair 1 0$, then Lemma~\ref{lemmaImtF}  yields that
$\fla p\kappa$ has $2^\kappa$ many ideals but only $\kappa$ many filters.
It is a relatively complemented simple lattice by Lemma~\ref{lemmaConLkpp} and it generates $\fvar p$ by Lemma~\ref{lemmageneratVp}.
Hence, the lattice  $L(p,2,\kappa):=\fla p\kappa$ satisfies
the requirements of \eqref{thmmainc}. 
Similarly, for $\pair  m n = \pair 0 1$, we let  
$L(p,3,\kappa):=\ula p\kappa$, which has $2^\kappa$ many ideals but only $\kappa$ many filters by Lemma~\ref{lemmaImtF}. This lattice generates $\fvar p$ by Lemma~\ref{lemmageneratVp}
and it has exactly three congruences by Lemma~\ref{lemmaConLkpp}, so  part \eqref{thmmainc}  
holds for $\pair m n= \pair 0 1$. 

We cannot apply Lemma~\ref{lemmaImtF}
to the undefined sum $\ula p\kappa\dot+\ula p\kappa$, since $\ula p\kappa$ has no largest element. Instead of recalling a more general concept of sums from Cz\'edli~\cite{czgtwodistb} and giving the easy generalization of Lemma~\ref{lemmaImtF} for it, it is more economic to take the $2^{m-1}$-element boolean lattice $B_{m-1}$; it will be needed only for $2\leq m\in\nnp$.
 It belongs to the folklore and we know also from Crawley~\cite[Theorem 3.2]{crawley} that $\Con(B_{m-1})\cong B_{m-1}$. 
Hence, $B_{m-1}$ has exactly $2^{m-1}$ congruences. 
Interrupting the proof, note that $\Con(M)$ is a boolean lattice and $|\Con(M)|$ is a power of 2  for \emph{every} finite modular lattice $M$; see, e.g., Gr\"atzer~\cite[Corollary 249 and Theorem 282]{ggfoundbook}.

Now, for the general case, we let
\begin{equation}
L(p,2^m\cdot 3^n,\kappa)=\underbrace{L(p,3,\kappa)\dot+\dots\dot+L(p,3,\kappa)}_{n\text{ copies}}\dot+ B_{m-1}\dot+
L(p,2,\kappa),
\label{eqdHdhBmT}
\end{equation}
where the summands $B_{m-1}$  and $L(p,2,m)$ are  present only if $m\geq 2$ and if $m\geq 1$, respectively.  Distributive lattices form a minimal (nontrivial) lattice variety; see, for example, Gr\"atzer~\cite[Page 421]{ggfoundbook}.
Combining this fact with Lemma~\ref{lemmageneratVp}, we obtain that every lattice identity $\Phi$ satisfied by $\fvar p$ holds in all summands occurring in \eqref{eqdHdhBmT}. Thus, $\Phi$ holds in $L(p,2^m\cdot 3^n,\kappa)$ by Lemma~\ref{lemmagldsM}. Conversely, if $\Phi$ holds in $L(p,2^m\cdot 3^n,\kappa)$, then it holds in all  summands of \eqref{eqdHdhBmT}. Hence, applying Lemma~\ref{lemmageneratVp} to a summand distinct from $B_{m-1}$, we obtain that $\Phi$ holds in $\fvar p$. Thus, $L(p,2^m\cdot 3^n,\kappa)$ generates $\fvar p$. 


Clearly, $L(p,2^m\cdot 3^n,\kappa)$ has $2^\kappa$ many ideals, since Lemma~\ref{lemmaImtF} applies to at least one of the summands in \eqref{eqdHdhBmT}.
Observe that 
since each of the \emph{finitely many} summands in \eqref{eqdHdhBmT} satisfies the Descending Chain Condition, so does $L(p,2^m\cdot 3^n,\kappa)$. It is clear by \eqref{eqdmhmNVT} that $|L(p,2^m\cdot 3^n,\kappa)|=\kappa$. Thus, Lemma~\ref{lemmadccfp} yields that $L(p,2^m\cdot 3^n,\kappa)$ has exactly $\kappa$ many filters. 
Using that we already know that $|\Con(L(p,3,\kappa))|=3$, $|\Con(L(p,2,\kappa))|=2$, and $|\Con(B_{m-1})|=2^{m-1}$, Lemma~\ref{lemmagldsM} implies that $L(p,2^m\cdot 3^n,\kappa)$ has exactly $2^m\cdot 3^n$ many congruences. 
This  completes the proof of part \eqref{thmmainc} and that of Theorem~\ref{thmmain}.
\end{proof}


\end{document}